\documentclass[11pt]{amsart}
\usepackage{CJK}
\usepackage{indentfirst}
\usepackage{amscd}
\usepackage{amsmath, amssymb}
\usepackage{amsfonts}
\usepackage{enumerate}
\usepackage{color}

\newcommand{\ddbar}{\sqrt{-1} \partial \overline{\partial}}
\newcommand{\ol}{\overline}

\newcommand{\ov}[1]{\overline{#1}}
\newcommand{\de}{\partial}
\newcommand{\dbar}{\overline{\partial}}
\newcommand{\omu}{\omega_{\underline{u}}}
\newcommand{\ti}{\tilde}

\newtheorem*{claim*}{Claim}

\begin{document}
%\begin{CJK}{GBK}{song}
\newcounter{theor}
\setcounter{theor}{1}
\newtheorem{claim}{Claim}
\newtheorem{theorem}{Theorem}[section]
\newtheorem{lemma}[theorem]{Lemma}
\newtheorem{corollary}[theorem]{Corollary}
\newtheorem{proposition}[theorem]{Proposition}
\newtheorem{prop}{Proposition}[section]
\newtheorem{question}{question}[section]
\newtheorem{defn}{Definition}[section]
\newtheorem{remark}{Remark}[section]

\numberwithin{equation}{section}

\title[Fully non-linear elliptic equations]{A remark for fully non-linear elliptic equations on compact almost Hermitian manifolds}

\author[L. Huang]{Liding Huang}
	\address{School of Mathematical Sciences, Xiamen University, Xiamen 361005, P. R. China}
	\email{huangliding@xmu.edu.cn}


\begin{abstract}
In this paper, we generalize the definition of sub-slope, introduced by Guo-Song, to almost Hermitian manifolds and prove the existence of solutions for  a general class of fully non-linear equations on compact almost Hermitian manifolds. As an application, we solve the complex Hessian quotient equation and the deformed Hermitian-Yang-Mills equation in the almost Hermitian setting.
\end{abstract}
\maketitle

\section{Introduction}\label{introduction}

The fully non-linear elliptic equation plays an important role in complex geometry and has been studied extensively with subsolution \cite{CNS85,CM21,Guan14,Szekelyhidi18,Yau78}. On the other hand, motivated by
differential geometry and mathematical physics, almost Hermitian manifolds have been researched extensively\cite{CHZ,CTW19,HZ,HL15,HZ}. In this paper, we consider the fully nonlinear elliptic equation with gradient terms on almost Hermitian setting. Suppose $\omega$ is  a real $(1,1)$-form on a compact almost Hermitian manifold $(M,\chi,J)$ of real dimension $2n$. For $u\in C^{2}(M)$, we denote
\[
\omega_{u}:= \omega+\ddbar{u}+Z(\partial u)
:= \omega+\frac{1}{2}(dJdu)^{(1,1)}
+Z(\partial u),
\]
where $Z(\partial u)$ denotes a smooth (1,1)-form depending on $\partial u$  linearly
and let $\lambda(u)=(\lambda_{1}(u),\ldots,\lambda_{n}(u))$ be the eigenvalues of $\omega_{u}$ with respect to $\chi$. Consider the following equation
\begin{equation}\label{nonlinear equation}
F(\omega_{u}) := f(\lambda_{1},\cdots,\lambda_{n}) = h,
\end{equation}
where  $h\in C^{\infty}(M)$, $f$ is a smooth symmetric function, For notational convenience, we sometimes denote $F(\omega_u)$, $\lambda_{i}(u)$ by $F(u)$, $\lambda_{i}$ respectively, when no confusion will arise. 
% The subsolution plays an important role in these results.  
%Later, Sz\'ekelyhidi \cite{Szekelyhidi18}  proved $C^{2,\alpha}$ estimate when $\mathcal{C}$-subsolution exists. In \cite{CM21}, Chu-McCleerey derived the real Hessian estimate independent of $\inf_{M}h$, which can be applied to the degenerate case of \eqref{nonlinear equation}. 

Consider the  open symmetric cone $\Gamma\subsetneq\mathbb{R}^{n}$    with vertex at the origin containing the positive orthant $\Gamma_{n}=\{\lambda\in\mathbb{R}^{n}:\lambda_{i}>0 \ \text{for $i=1,\ldots,n$}\}$.  We assume that $f$ is defined on $
\Gamma$. Furthermore, suppose that
\begin{enumerate}[(i)]\setlength{\itemsep}{1mm}
\item \begin{equation}\label{elliptic}
f_{i}=\frac{\de f}{\de\lambda_{i}}>0
\end{equation}
for all $i$ and $f$ is concave,
\item $\sup_{\partial\Gamma}f<\inf_{M}h$,
\item For any $\sigma<\sup_{\Gamma} f$ and $\lambda\in \Gamma$, we have $\lim_{t\rightarrow \infty}f(t\lambda)>\sigma$,
\end{enumerate}
where
\[
\sup_{\de\Gamma}f = \sup_{\lambda'\in\de\Gamma}\limsup_{\lambda\rightarrow\lambda'}f(\lambda).
\]
Many geometric equations are of the form \eqref{nonlinear equation}, such as complex Monge--Amp\`ere equation  complex Hessian equation, complex Hessian quotient equation  and the Monge--Amp\`ere equation for $(n-1)$-plurisubharmonic functions.
On almost Hermitian manifolds, Chu-Huang-Zhang \cite{CHZ}   proved the $C^{2,\alpha}$-estimates for a class of nonlinear elliptic equations and Huang-Zhang \cite{HZ}  generalized their result to the equation involving gradient terms. Recently, Guo-Song \cite{GS} introduced the sub-slope which can be used to solve a class of fully nonlinear elliptic equation on Hermitian manifolds.  Here we generalized their result to almost Hermitian manifolds. 

    Now we state our main results.
    \begin{theorem}\label{main}
    Let $(M, g)$ be a closed almost Hermitian manifold with real dimension 2n. Then the following are equivalent
    \begin{enumerate}
         \item There exists a smooth solution $u\in \mathcal{E}$ solving the following equation
         \begin{equation}\label{nonlinear equation subslope}
            F(\omega_u)=h+\sigma,
        \end{equation}
        where $\sigma$ are sub-slope defined in Definition \ref{subslope}.
        \item There exists a $\mathcal{C}$-subsolution $\underline{u}\in \mathcal{E}$ for equation \eqref{nonlinear equation subslope}.
\item There exists $u\in \mathcal{E}$ satisfying  $\sigma<\min (F_{\infty}(\omega_u)-h),$
where  $F_{\infty}$ are defined in \eqref{f infty}.
%\item There exists $\underline{u},\overline{u}\in \mathcal{E}$ satisfying \begin{equation}
   % \max_{M}(F(\overline{u})-h)<\min_{M} (F_{\infty}(\underline{u})-h)
%\end{equation}
    \end{enumerate}
        
  %  Furthermore, if $u\in \mathcal{E}$ solves equation \eqref{nonlinear equation}, u is unique up to a constant and 
   % \[F(u)=h+\sigma,\]
 %   where $\sigma$ is the sup-slope.
    \end{theorem}
iii
 We now discuss the proof of Theorem \ref{main}. Using Guo-Song's argument \cite{GS}, we prove the subsolution can be preserved in the continuity method  with sub-slope. Recently, Lin \cite{Lin25} also generalize Guo-Song's argument \cite{GS} to almost Hermitian manifolds.   Compared with Guo-Song's and Lin's setup, we modify their definition (see Definition \ref{subslope}) and we do not make the assumption that the right hand is positive and use the assumption (iii) to replace $\lim_{t\rightarrow \infty} f(t \lambda)=\infty$.
 
We can consider the complex quotient Hessian equation. When $(M,\omega)$ is K\"ahler, Song-Weinkove \cite{SW08} solved  the complex Hessian quotient equation for $k=n$, $l=n-1$ (J-equation) and $h$ is constant. Fang-Lai-Ma  
 \cite{FLM11} generalized this result to $k=n$ and general $l$. and Sz\'ekelyhidi \cite{Szekelyhidi18} considered general $k$ and $l$. Sun obtained analogous results when $\omega$ is Hermitian and $h$ is not constant \cite{Sun16,Sun17a,Sun17b} (see also \cite{Li14,GS15}). Guo-Song \cite{GS}  gave a sufficiency and necessary condition for J-equation on Hermitian manifolds, using the sub-slope. When $k=n$ and $1\leq l\leq n-1$, J. Zhang \cite{ZJ21} proved a priori estimates in the almost Hermitian setting. However, there are few results for the existence of the complex Hessian quotient equations in the almost Hermitian seting.
  
Since the complex quotient Hessian equation satisfies condition (i)-(iii), using Theorem \ref{main}, we can solve the complex quotient Hessian equation. For the definitions of $k$-positivity and $\Gamma_{k}(M,\chi)$, we refer the reader to Subsection \ref{cone}.

\begin{theorem}\label{complex Hessian equation}
Let $(M,\chi,J)$ be a compact almost Hermitian manifold of real dimension $2n$ and $\omega$ be a smooth $k$-positive real $(1,1)$-form. Assume h is a smooth function. Suppose there exists a $\mathcal{C}$-subsolution for equations
\begin{equation}\label{complex Hessian equation 1}
\begin{cases}
\ \omega_{u}^{k}\wedge \chi^{n-k}=e^{h+\sigma}\omega_{u}^{l}\wedge \chi^{n-l}, \\[1mm]
\ \omega_{u}\in\Gamma_{k}(M,\chi), \\[0.5mm]
\ \sup_{M}u = 0,
\end{cases}
\end{equation}
where  $1\leq l <k\leq n$. Then there exists a unique solution $u\in C^{\infty}(M)$ for equation \eqref{complex Hessian equation 1}. Here  $\sigma$ are defined by
\begin{equation}
\log \sigma=\min_{\omega_{u'}\in \Gamma_k}\max_{M}\{\log(\frac{\omega_{u'}^{k}\wedge \chi^{n-k}}{\omega_{u'}^{l}\wedge \chi^{n-l}})- h\}.
\end{equation}
\end{theorem}

For the deformed Hermitian-Yang-Mills (dHYM) equation
\begin{equation}\label{DHYM}
\phi(\mu)=\sum_{i=1}^{n}\mathrm{arctan}\,\lambda_{i}=h,\qquad h\in C^{\infty}(M).
\end{equation}
%where $h\in \big[(n-2)\frac{\pi}{2}+\delta,n\frac{\pi}{2}\big)$ is a smooth function and $0<\delta<\frac{\pi}{2}$ is a constant. 
The equation \eqref{DHYM} is hypercritical (resp. supercritical) if $h\in (\frac{(n-1)\pi}{2}, \frac{n\pi}{2})$ (resp. $h\in (\frac{(n-2)\pi}{2}, \frac{n\pi}{2})$).  Jacob-Yau \cite{JY} showed  the existence of solution for dimension 2,  and for general dimensions they gave a sufficient condition for the existence for solutions. In general dimensions, the equation \eqref{DHYM}
was solved by Collins-Jacob-Yau \cite{CJY}  under the existence of $\mathcal{C}$-subsolutions. We refer the reader to interesting papers such
as \cite{HZ,HZZ,Lin20,PV,PV1}. However, the existence of solutions to dHYM equations has not yet been proven on almost Hermitian manifolds. Now we provide a necessary and sufficient condition for the existence of solutions for dHYM equations on almost Hermitian manifolds.

\begin{theorem}\label{dHYM}
Let $(M,\chi,J)$ be a compact almost Hermitian manifold of real dimension $2n$ and $\omega$ be a smooth positive real $(1,1)$-form.  Assume h is a smooth function with $h\in (\frac{(n-2)\pi}{2}, \frac{n\pi}{2})$. Suppose there exists a $\mathcal{C}$-subsolution for equations
\begin{equation}\label{dhym 1}
\begin{cases}
\sum_{i=1}^{n}\mathrm{arctan}\,\lambda_{i}=h+\sigma, \\[1mm]
\ \sup_{M}u = 0.
\end{cases}
\end{equation}
 Then there exists a unique solution $u\in C^{\infty}(M)$ for equation \eqref{dhym 1}. Here $\sigma$ are defined by
 \begin{equation}
     \tan\sigma=\min_{\sum_{i=1}^{n}\mathrm{arctan}\,\lambda_{i}(u)\in(\frac{(n-2)\pi}{2}, \frac{n\pi}{2})}\max_{M} \tan (\sum_{i=1}^{n}\mathrm{arctan}\,\lambda_{i}(u)-h).
 \end{equation}
\end{theorem}

Although the deformed Hermitian-Yang-Mills equations do not satisfy the condition (iii), the proof is similarly to Theorem \ref{main}. In the following we use F(u) or $f(\lambda)$ to denote both the equation \eqref{nonlinear equation} satisfying 
(i),(ii) and (iii) and $F(u)=f(\lambda)=\sum_{i}\arctan\lambda_i$ when no confusion will arise.

The paper is organized as follows. In Section \ref{preliminaries}, we will introduce some notations, recall definition of $\mathcal{C}$-subsolution and sub-slope.  In Section 3, we will prove Theorem \ref{main} and Theorem \ref{complex Hessian equation}.

\section{Preliminaries}\label{preliminaries}

\subsection{Notations}
We first recall the definition of   $(p,q)$-form and operators $\de$, $\dbar$ on almost Hermitian manifolds $(M,\chi,J)$ (see e.g. \cite[p. 1954]{CTW19}). We denote
\[
A^{1,1}(M) = \big\{ \alpha: \text{$\alpha$ is a smooth real (1,1)-forms on $(M,J)$} \big\}
\]
and
\[
\ddbar u = \frac{1}{2}(dJdu)^{(1,1)}
\]
for any $u\in C^{\infty}(M)$.

For any point $x_{0}\in M$, suppose $\{e_{i}\}_{i=1}^{n}$ is a local unitary $(1,0)$-frame with respect to $\chi$ near $x_{0}$ and $\{\theta^{i}\}_{i=1}^{n}$ are its dual coframe. It follows
\[
\chi=\sqrt{-1}\delta_{ij}\theta^{i}\wedge\ov{\theta}^{j}.
\]
Using $\{\theta^{i}\}_{i=1}^{n}$, assume
\[
\omega = \sqrt{-1}g_{i\ov{j}}\theta^{i}\wedge\ov{\theta}^{j}, \quad
\omega_{u} = \sqrt{-1}\ti{g}_{i\ov{j}}\theta^{i}\wedge\ov{\theta}^{j},
\]
where (see e.g. \cite[(2.5)]{HL15})
\[
\ti{g}_{i\ov{j}} = g_{i\ov{j}}+(\de\dbar u)(e_{i},\ov{e}_{j})+Z_{i\bar j}
= g_{i\ov{j}}+e_{i}\overline{e}_{j}(u)-[e_{i},\overline{e}_{j}]^{(0,1)}(u)+Z_{i\bar j}.
\]
Here $[e_{i}, \bar{e}_{j}]^{(0,1)}$ is the $(0,1)$-part of the Lie bracket $[e_{i}, \bar{e}_{j}]$. Define
\begin{equation*}
F^{i\overline{j}}=\frac{\partial F}{\de\ti{g}_{i\ov{j}}}, \quad
F^{i\overline{j},k\overline{l}}=\frac{\partial^{2}F}{\de\ti{g}_{i\ov{j}}\de\ti{g}_{k\ov{l}}}.
\end{equation*}
In addition, there exists $\{e_{i}\}_{i=1}^{n}$ such that $\tilde{g}_{i\overline{j}}(x_{0})=\delta_{ij}\tilde{g}_{i\overline{i}}(x_{0})$.
We denote $\tilde{g}_{i\overline{i}}(x_{0})$ by $\lambda_{i}$ and assume
\begin{equation}\label{mu order}
\lambda_{1}\geq\lambda_{2}\geq\cdots\geq\lambda_{n}.
\end{equation}
At $x_{0}$, we have (see e.g. \cite{Spruck05})
\begin{equation}\label{second derive of F}
F^{i\ov{j}} = \delta_{ij}f_{i}.
\end{equation}
 Using \eqref{mu order}, we obtain (see e.g. \cite{Spruck05})
\begin{equation}\label{F ii 1}
F^{1\ov{1}} \leq F^{2\ov{2}} \leq \cdots \leq F^{n\ov{n}}.
\end{equation}
\iffalse
Finally, the linearized operator of equation \eqref{nonlinear equation} is given by
\begin{equation}\label{L}
L:=\sum_{i,j}F^{i\bar{j}}(e_{i}\bar{e}_{j}-[e_{i},\bar{e}_{j}]^{0,1}).
\end{equation}
Since $[e_{i},\bar{e}_{j}]^{(0,1)}$ is a first order differential operator,  $L$ is a second order elliptic operator.

\fi

\subsection{$\mathcal{C}$-subsolution}
First we recall the definition of $\mathcal{C}$-subsolution introduced by Sz\'ekelyhidi \cite{Szekelyhidi18}, based on Guan \cite{Guan14}. 
\begin{defn}[Definition 1 of \cite{Szekelyhidi18}]\label{sub}
We say that a smooth function $\underline{u}:M\rightarrow\mathbb{R}$ is a $\mathcal{C}$-subsolution of \eqref{nonlinear equation} or \eqref{DHYM} if at each point $x\in M$, the set
\begin{equation*}
\left\{\lambda\in \Gamma: f(\lambda)=h(x) \text{\ and\ }  \lambda-\lambda(\underline{u})\in \Gamma_{n}\right\}
\end{equation*}
is bounded, where $\lambda(\underline{u})=(\lambda_{1}(\underline{u}),\ldots,\lambda_{n}(\underline{u}))$ denote the eigenvalues of $\omega+\ddbar\underline u$ with respect to $\chi$.
\end{defn}

Set
\begin{equation*}
    \mathcal{E}=\{u\in C^{\infty}(M): \lambda(\omega_{u})(p)\in \Gamma, \text{for any p}\in M\}.
\end{equation*}
Consider the function
\begin{equation*}
    f_{\infty,i}:\Gamma\rightarrow \mathbb{R},
\end{equation*}
where
\begin{equation*}
    f_{\infty, i}(\lambda_1,\cdots, \lambda_n)=\lim_{R\rightarrow \infty}f(\lambda_1,\cdots, \lambda_{i-1}, R, \lambda_{i+1}, \cdots, \lambda_n),
\end{equation*}
for $i=1,2,\cdots, n$ and
\begin{equation*}
    f_{\infty}(\lambda)=\min_{i=1,\cdots,n}f_{\infty,i}(\lambda).
\end{equation*}
We also denote
\begin{equation}\label{f infty}
    F_{\infty}(u)=f_{\infty}(\lambda(\omega_u)).
\end{equation}

For deformed Hermitian-Yang-Mills equation, Collins-Jacob-Yau proved the following results.
\begin{proposition}\label{dhym subsolution}
    The function $\underline{u}$ is a $\mathcal{C}$-subsolution of the equation\eqref{DHYM}
    if and only if $\sum_{i\neq k}\arctan\lambda_i(x)>h-\frac{\pi}{2}$ for any $x\in M$, $k=1,\cdots,n$.
\end{proposition}

For equation \eqref{nonlinear equation} satisfying (i)-(iii), a equivalent characterization of $\mathcal{C}$-subsolution also proved by Sz\'ekelyhidi \cite{Szekelyhidi18}. 
\begin{proposition}\label{equivalent}
   The function $\underline{u}$ is a $\mathcal{C}$-subsolution of the equation $f(\lambda(\omega_u))=h(x)$
    if and only if $f_{\infty}(\lambda(\omu))(x)>h(x)$ for any $x\in M$.
\end{proposition}
\begin{remark}
    Using Proposition \ref{dhym subsolution}, it is not difficult to see that Proposition \ref{equivalent} is also hold for the deformed Hermitian-Yang-Mills equations.
\end{remark}

In addition, $f_{\infty}$ have the following properties (see \cite{NT,GS}):
\begin{proposition}
The function $f_{\infty}$ satisfies
\begin{enumerate}
    \item  either $f_{\infty}\equiv\infty$ in $\Gamma$; \\
     \item or $f_{\infty}(\lambda)$ is bounded for each $\lambda\in \Gamma$. 
     \end{enumerate}
\end{proposition}
   \begin{remark}
   It is easy to see that for deformed Hermitian-Yang-Mills equation,  $f_{\infty}(\lambda)$ is bounded.
\end{remark}

Using the $\mathcal{C}$-subsolution, Huang-Zhang \cite[Corollary 1.4 or Proposition 3.11]{HZ} proved the a priori estimates on almost Hermitian manifolds: 
\begin{theorem}\label{main estimate}
Let $(M,\chi,J)$ be a compact almost Hermitian manifold of real dimension $2n$ and $\underline{u}$ is a $\mathcal{C}$-subsolution of \eqref{nonlinear equation} or \eqref{DHYM}. Suppose that $u$ is a smooth solution of (\ref{nonlinear equation}) or \eqref{DHYM}. Then for any $\alpha\in(0,1)$, we have the following estimate
\begin{equation*}
	\| u\|_{C^{2,\alpha}(M,\chi)}\leq C,
\end{equation*}
where $C$ is a constant depending only on  $\alpha$, $\underline{u}$, $h$, $\omega$, $f$, $\Gamma$ and $(M,\chi,J)$.
\end{theorem}

Recently, Guo-Song \cite{GS} introduced the sub-slope such  that the subsolutions can be preserved along continuity path in continuity method on Riemannian manifolds or Hermitian manifolds. Here we modify their definition such that we  do not need  that the right hand is positive.
\begin{defn}\label{subslope}
    The sub-slope $\sigma$ for equation  \eqref{nonlinear equation} or \eqref{DHYM} associated to the cone $\Gamma$ is defined to be 
    $$\sigma=\inf_{u\in \mathcal{E}}\max_{M}(F(u)-h).$$
\end{defn}
\begin{remark}
   The constant $\sigma$ has a lower bound: $\sigma\geq \inf_M F(\omega)-\sup_{M}h$. In fact for any function $u\in \mathcal{E}$, assume $u$ achieves a minimum at $x_0$.Then 
    \begin{equation*}
        F(\omega_u)(x_0)\geq F(\omega)(x_0).
    \end{equation*}
    It follows that
    \begin{equation*}
    \begin{split}
         \max_{M}(F(\omega_u)-h)\geq &F(\omega_u)(x_0)-h(x_0)\\
    \geq & F(\omega)(x_0)-\sup h\\
    \geq & \inf_M F(\omega)-\sup_{M}h.
    \end{split}
    \end{equation*}
\end{remark}

\subsection{G{\aa}rding cone}\label{cone}
 For $1\leq k\leq n$ and  any $\lambda=(\lambda_{1},\lambda_{2},\cdots,\lambda_{n})\in\mathbb{R}^{n}$, we denote
\begin{equation*}
\sigma_{k}(\lambda) =\sum_{1\leq i_{1}<\cdots<i_{k}\leq n}\lambda_{i_{1}}\lambda_{i_{2}}\cdots\lambda_{i_{k}}
\end{equation*}
and
\begin{equation*}
\Gamma_{k} = \{ \lambda\in\mathbb{R}^{n}: \text{$\sigma_{i}(\lambda)>0$ for $i=1,2,\cdots,k$} \}.
\end{equation*}
Namely, $\sigma_{k}$ and $\Gamma_{k}$ are the $k$-th elementary symmetric polynomial and the $k$-th G{\aa}rding cone on $\mathbb R^n$, respectively.
We can   extend the above definitions  to almost Hermitian manifold $(M,\chi,J)$ as follows:
\begin{equation*}
\sigma_{k}(\alpha)=\binom{n}{k}
\frac{\alpha^{k}\wedge\chi^{n-k}}{\chi^{n}}, \quad \text{for any $\alpha\in A^{1,1}(M)$},
\end{equation*}
and
\begin{equation*}
\Gamma_{k}(M,\chi) = \{ \alpha\in A^{1,1}(M): \text{$\sigma_{i}(\alpha)>0$ for $i=1,2,\cdots,k$} \}.
\end{equation*}
We say that $\alpha\in A^{1,1}(M)$ is $k$-positive if $\alpha\in\Gamma_{k}(M,\chi)$.
\section{Proof of Theorem \ref{main}}

 By Proposition \ref{equivalent},  (2) is equivalent to (3) in Theorem \ref{main}. We only need to prove the following theorem. In this section, our results are hold for both equation \eqref{nonlinear equation} and the deformed Hermitian-Yang-Mills equation \eqref{DHYM}.
 \begin{theorem}\label{existence}
 Suppose there exists $\underline{u}\in \mathcal{E}$ satisfying  
 \begin{equation}\label{subslope subsolution}
 \sigma<\min (F_{\infty}(\underline{u})-h).
 \end{equation}
 then
              we have a smooth solution $u\in \mathcal{E}$ solving the following equation
         \begin{equation}
            F(u)=h+\sigma.
        \end{equation}
 \end{theorem}
 By definition of sub-slope $\sigma$ and \eqref{subslope subsolution}, there exists a function
  $\bar u\in \mathcal{E}$  satisfying
\begin{equation}\label{bar c}
    \sigma\leq\bar c=\max_{M}(f(\lambda(\omega_{\bar{u}}))-h)<\min(f_{\infty}(\lambda(\omega_{\underline{u}}))-h).
\end{equation}
Then there exists $\delta$ such that
\begin{equation}\label{delta}
    \bar c<\bar c+\delta\leq \min_{M}(f_{\infty}(\lambda(\omega_{\underline{u}}))-h).
\end{equation}
By letting $\bar h=F(\bar u)$. Then $\bar h\leq h+\bar c$. In conclusion, we have the following Lemma
\begin{lemma}\label{bar h}
   $ \bar h\leq h +\bar c$
\end{lemma}
We consider the family of equations
\begin{equation}\label{continuity equation}
    F(\bar u+\phi_t)=h_t+c_t, h_t=(1-t)\bar h+th, t\in [0,1],
\end{equation}
where $\bar u+\phi_t\in \mathcal{E},\sup\phi_t=0$ .
Let 
$$I=\{t\in [0,1]\,|\,\text{ there exists a pair $(\phi_t, c_t)$ solving \eqref{continuity equation} at }\, t \}.$$
It is obvious that $0\in I$. It suffices to prove $I$ is open and closed. First we prove the closeness of $I$.

\subsection{Closeness}
In the following, we prove the closeness. Now we prove the constant $c_t$ is bounded.
\begin{lemma}\label{ct}
    There exists $C>0$ such that for all $t\in I$,
    \begin{equation*}
        -C\leq c_t\leq t\bar c.
    \end{equation*}
\end{lemma}

\begin{proof}
Assume that $p_t\in M$ is the maximal point of $\phi_t$. By the maximum principle, we have, at the maximal point $p_t\in M$ of $\phi_t$,
\begin{equation*}
    (1-t)\bar h(p_t)+th(p_t)+c_t=F(\bar u+\phi_t)(p_t)\leq F(\bar u)(p_t)=\bar h(p_t).
\end{equation*}
Therefore by Lemma \ref{bar h}, $c_t\leq t(\bar h-h )(p_t)\leq t\bar c$.
At the minimal pint $q_t$ of $\phi_t$, we have
\begin{equation*}
    (1-t)\bar h(q_t)+th(q_t)+c_t=F(\bar u+\phi_t)(q_t)\geq F(\bar u)(q_t)=\bar h(q_t),
\end{equation*}
which implies $c_t\geq -C$.

\end{proof}

\iffalse

\begin{proposition}
\color{red}{For any $h\in C^{\infty}$ and $C\in \mathbb{R}$ the sub-slope satisfies    
\begin{equation}
        \sigma(M, \chi, \omega,f, h+C )=\sigma(M, \chi, \omega,f, h )+C.
    \end{equation}
    Moreover, for any $h_1,h_2\in C^{\infty} $  , we have}
    \begin{equation}
         -\max_{M}|h_2-h_1|\leq \sigma(M, \chi, \omega,f, h_2)- \sigma(M, \chi, \omega,f, h_1 )\leq \max_{M}|h_2-h_1|.
    \end{equation}
\end{proposition}
\fi
\begin{proposition}
    Let $\delta$ be the fixed constant in \eqref{delta} and $\underline{u}\in \mathcal{E}$ be the subsolution defined in Theorem \ref{existence}. Then for any $t\in \mathcal{T}$, we have
    \begin{equation}
        f_{\infty}(\lambda(\omega_{\underline{u}}))-h_t\geq c_t+\delta.
    \end{equation}
\end{proposition}
\begin{proof}
By Lemma \ref{bar h}, we have
\begin{equation}
    -(1-t)\bar h\geq -(1-t)h-(1-t)\bar c.
\end{equation}
Then, by directly calculation, we have
\begin{equation}
    \begin{split}
        &f_{\infty}(\lambda(\omega_{\underline{u}}))-h_t-c_t\\
        =&f_{\infty}(\lambda(\omega_{\underline{u}}))-(1-t)\bar{h}-th-c_t\\
        \geq &\min_{M}(f_{\infty}(\lambda(\omega_{\underline{u}}))-h-(1-t)\bar c-c_t)\\
        \geq& -c_t-(1-t)\bar c+\delta+\max_{M}(F(\bar u)-h)\\
        =& -c_t-(1-t)\bar c+\bar c+\delta\\
        \geq& \delta.
    \end{split}
\end{equation}
Here in the second  inequality we used \eqref{delta} and in the last inequality we used Lemma \ref{ct}.
\end{proof}

Using Proposition  $\ref{equivalent}$, we immediately obtain 
\begin{proposition}\label{subslution t}
    $\underline{u}$ is a $\mathcal{C}$-subsolution of \eqref{continuity equation} for all $t\in I$.
\end{proposition}

Using Proposition \ref{subslution t} and Theorem \ref{main estimate}, we have $\|\phi_t\|_{C^{2,\alpha}(M)}\leq C$ which implies $I$ is closed.

\iffalse
\begin{lemma}
    If $\underline{u}$ is a $\mathcal{C}$-subsolution, then we have
    \begin{equation}
        F_{\infty}(\underline{u})-h> \sigma.
    \end{equation}    
\end{lemma}
\begin{proof}
   Suppose $(\lambda(\omu)_1,\cdots,\lambda(\omu)_n)$ be the eigenvalue of $\omu$ with respect to $\chi$. By definition of $\mathcal{C}$-subsolution, there exists constant $C_0$ such that for any
    $x\in M$, 
\begin{equation}\label{bounded subsolution}
\lambda\in \left\{\lambda\in \Gamma: f(\lambda)=h(x) \text{\ and\ }  \lambda-\lambda(\underline{u})\in \Gamma_{n}\right\}
\end{equation}
   we have $|\lambda|\leq C_0$.
  Suppose    \begin{equation}
        F_{\infty}(\underline{u})-h\leq  \sigma.
    \end{equation} 
Consider $(\lambda(\omu)_1+|\lambda(\omu)_1|+C_0,\lambda(\omu)_{2},\cdots,\lambda(\omu)_n+\epsilon)$   
Then  $\lambda(\omu)_1+|\lambda(\omu)_1|+C_0,\lambda(\omu)_{2},\cdots,\lambda(\omu)_n+\epsilon)\in \lambda(\omu)+\Gamma_n$
and
\begin{equation*}
f(\lambda(\omu)_1+|\lambda(\omu)_1|+C_0,\lambda(\omu)_{2},\cdots,\lambda(\omu)_n+\epsilon)< h+\sigma.
\end{equation*}
Using the condition (iii), there exists a $t>0$ such that
\begin{equation*}
f(\lambda(\omu)_1+|\lambda(\omu)_1|+C_0+t,\lambda(\omu)_{2}+t,\cdots,\lambda(\omu)_n+\epsilon+t)=h+\sigma.
\end{equation*}
It is contract with \eqref{bounded subsolution}.
\end{proof}
\fi

\subsection{Openness}
\begin{lemma}
    I is open.
\end{lemma}
\begin{proof}
   Let  $(\phi_{\hat{t}},c_{\hat{t}})$ be the solution of   \eqref{continuity equation} at $\hat{t}$. We only need to to show that there exists a pair $(\phi_{t},c_{t})\in C^{\infty}(M)\times\mathbb{R}$ solving \eqref{continuity equation} for $t$ which is close to $\hat{t}$.  	

SUpposethat the linearized operator of \eqref{continuity equation} at $\phi_{\hat{t}}$ is $L_{u_{\hat{t}}}(\psi)$.
 Using the maximum principle,
\begin{equation}\label{ker L1}
\mathrm{Ker}(L_{u_{\hat{t}}})=\{\text{constant}\}.
\end{equation}
Suppose  $L_{\phi_{\hat{t}}}^{*}$ are the $L^{2}$-adjoint operator of $L_{\phi_{\hat{t}}}$ with respect to the volume form
\[
\mathrm{dvol}=\chi^{n}.
\]
Since the operator $L_{u_{\hat{t}}}$ is elliptic,  the index is zero.
By the Fredholm alternative, there is a non-negative function $\xi$ such that
\begin{equation}\label{ker L*}
\mathrm{Ker}(L_{u_{\hat{t}}}^{*})=\text{Span}\big\{\xi\big\}.
\end{equation}
 From the strong maximum principle, we have $\xi>0$. Up to a constant, we may assume
\[
\int_{M}\xi\, \mathrm{dvol}=1.
\]

To prove the openness of $I$, we define the space  by
\[
\mathcal{U}^{2,\alpha}:=\Big\{\phi\in C^{2,\alpha}(M):
\text{$\omega_{\phi}\in\Gamma(\chi)$ and
$\int_{M}\phi\cdot\xi \, \mathrm{dvol}=0$}
\Big\}.
\]
Then the tangent space of $\mathcal{U}^{2,\alpha}$ at $\phi_{\hat{t}}$ is given by
\[
T_{u_{\hat{t}}}\,\mathcal{U}^{2,\alpha}
:=\Big\{\psi\in C^{2,\alpha}(M): \int_{M}\psi\cdot\xi \, \mathrm{dvol}=0\Big\}.
\]
Let us consider the map
\begin{equation*}
G(\phi,c)=F(\phi)-c,
\end{equation*}
which maps $\mathcal{U}^{2,\alpha}\times\mathbb{R}$ to $C^{\alpha}(M)$.
It is clear that the linearized operator of $G$ at $(u_{\hat{t}},\hat{t})$ is given by
\begin{equation}\label{linear operator}
(L_{u_{\hat{t}}}-c): T_{u_{\hat{t}}}\,\mathcal{U}^{2,\alpha}\times \mathbb{R}\longrightarrow  C^{\alpha}(M).
\end{equation}
On the one hand, for any $h\in C^{\alpha}(M)$, there exists a constant $c$ such that
\[
\int_{M}(h+c)\cdot\xi \, \text{dvol}_{k}=0.
\]
By \eqref{ker L*} and the Fredholm alternative, we can find a real function $\psi$ on $M$ such that
\[
L_{u_{\hat{t}}}(\psi)=h+c
\]
Hence, the map $L_{u_{\hat{t}}}-c$ is surjective. On the other hand, suppose that there are two pairs $(\psi_{1},c_{1}),(\psi_{2},c_{2})\in T_{u_{\hat{t}}}\,\mathcal{U}^{2,\alpha}\times \mathbb{R}$ such that
\[
L_{u_{\hat{t}}}(\psi_{1})-c_{1}
= L_{u_{\hat{t}}}(\psi_{2})-c_{2}.
\]
It then follows that
\[
L_{u_{\hat{t}}}(\psi_{1}-\psi_{2}) = c_{1}-c_{2}.
\]
Applying the maximum principle twice, we obtain $c_{1}=c_{2}$ and $\psi_{1}=\psi_{2}$. Then $L_{u_{\hat{t}}}-c$ is injective.

Now we conclude that $L_{u_{\hat{t}}}-c$ is bijective. By the inverse function theorem, when $t$ is close to $\hat{t}$, there exists a pair $(u_{t},c_{t})\in\mathcal{U}^{2,\alpha}\times\mathbb{R}$ satisfying
\begin{equation*}
 G(u_{t},c_{t}) = th+(1-t)h_{0}.
\end{equation*}
The standard elliptic theory shows that $u_{t}\in C^{\infty}(M)$. Then $S$ is open.
\end{proof}

In conclusion, we prove there exists a solution for equation \eqref{nonlinear equation} if we have a $\mathcal{C}$-subsolution. Now we prove the constant $c_t$ on the right hand of equation is equal to the sub-slope.

\begin{lemma}
    Suppose $u\in \mathcal{E}$ solves equations \eqref{nonlinear equation}. Then 
     $F(u)=h+\sigma$ where $\sigma$ is the sup-slope for \eqref{nonlinear equation}.
\end{lemma}

\begin{proof}
    By definition of $u$ and $\sigma$,
    \begin{equation*}
        \sigma\leq \max_{M}(F(u)-h)=(F(u)-h).
    \end{equation*}
    Suppose $\sigma<(F(u)-h)$. Then there exists $u'\in \mathcal{E}$ such that
    \begin{equation*}
        \sigma\leq \max_{M}(F(u')-h)< (F(u)-h).
    \end{equation*}
    Applying the maximum principle at the minimum point of $\phi=u'-u$, we have
    \begin{equation*}
        F(u)-h\leq F(u')-h\leq \max(F(u')-h)<(F(u)-h).
    \end{equation*}
    It is a contraction.
\end{proof}

\end{document}